\documentclass[12pt]{article}
\usepackage{amsmath,amssymb,amsfonts,amsthm,graphicx,float,fancyhdr}
\usepackage[a4paper,left=3cm,right=3cm,top=3cm,bottom=3cm]{geometry}
\usepackage{hyperref}
\usepackage[pagewise]{lineno}
\newtheorem{theorem}{Theorem}[section]
\newtheorem{proposition}[theorem]{Proposition}
\newtheorem{corollary}[theorem]{Corollary}
\newtheorem{lemma}[theorem]{Lemma}
\newtheorem{remark}[theorem]{Remark}

\numberwithin{equation}{section}

\title{Some asymptotic results for the continued fraction expansions with odd partial quotients}
\author{
    Gabriela Ileana Sebe\footnote{e-mail: igsebe@yahoo.com.} \\
    \emph{\small Politehnica University of Bucharest, Faculty of Applied Sciences},\\
    \emph{\small Splaiul Independentei 313, 060042, Bucharest, Romania} and \\
    \emph{\small Institute of Mathematical Statistics and Applied Mathematics}, \\
     \emph{\small Calea 13 Sept. 13, 050711 Bucharest, Romania}
    and\\
    Dan Lascu\footnote{e-mail: lascudan@gmail.com, corresponding author}\nonumber \\
    \emph{\small Mircea cel Batran Naval Academy, 1 Fulgerului, 900218 Constanta,
    Romania}
    }
\begin{document}
\maketitle
\pagestyle{fancy}
\lhead{G.I. Sebe and D. Lascu}
\rhead{Some asymptotic results for OCF}
\begin{abstract}We present and develop different approaches to study the asymptotic behavior of the distribution functions in the odd continued fractions case.
Firstly, by considering the transition operator of the Markov chain associated with these expansions on a certain Banach space of complex-valued functions of bounded variation we make a brief survey of the solution in the Gauss-Kuzmin-type problem.
Secondly, we use the method of Sz\"usz to obtain a similar asymptotic result and to give a good estimate of the convergence rate involved.

\end{abstract}

\section{Introduction}
\label{Sec:1}

We define OCF$(x)$ the continued fraction expansions with odd partial quotients of $x \in [0, 1]$ as follows.
Let us partition the unit interval $[0, 1]$ into intervals
$\left( \frac{1}{2k}, \frac{1}{2k-1} \right]$, for $k=1, 2, 3, \ldots$, and
$\left( \frac{1}{2k-1}, \frac{1}{2k-2} \right]$, for $k=2, 3, 4, \ldots$, and
consider the transformation $T:[0, 1] \to [0, 1]$ defined by $T(0)= 0$ and
\begin{equation}
T(x) :=
\left\{
\begin{array}{ll}
{\displaystyle \frac{1}{x}- \left\lfloor\frac{1}{x}\right\rfloor},& { x \in  \displaystyle \bigcup_{k \geq 1} \left( \frac{1}{2k}, \frac{1}{2k-1} \right]}\\
{\displaystyle 1 - \left( \frac{1}{x}- \left\lfloor\frac{1}{x}\right\rfloor \right)},& { x \in \displaystyle \bigcup_{k \geq 2} \left( \frac{1}{2k-1}, \frac{1}{2k-2} \right]}.
\end{array}
\right. \label{1.1}
\end{equation}
With auxiliary functions
\begin{equation}
\varepsilon_1(x) =
\left\{
\begin{array}{lll}
1,  & x \in \displaystyle \bigcup_{k \geq 1} \left( \frac{1}{2k}, \frac{1}{2k-1} \right], \\
-1, & x \in \displaystyle \bigcup_{k \geq 2} \left( \frac{1}{2k-1}, \frac{1}{2k-2} \right]
\end{array} \right. \label{1.2}
\end{equation}
and
\begin{equation}
a_1(x) =
\left\{
\begin{array}{lll}
\displaystyle \left\lfloor\frac{1}{x}\right\rfloor   & x \in \displaystyle \bigcup_{k \geq 1} \left( \frac{1}{2k}, \frac{1}{2k-1} \right], \\
1+\displaystyle\left\lfloor\frac{1}{x}\right\rfloor & x \in \displaystyle \bigcup_{k \geq 2} \left( \frac{1}{2k-1}, \frac{1}{2k-2} \right]
\end{array} \right. \label{1.3}
\end{equation}
we arrive at
\begin{equation*}
T(x) = \varepsilon_1(x) \left( \frac{1}{x} - a_1(x) \right), \quad x \in (0, 1). \label{1.4}
\end{equation*}
Note that
\begin{equation*}
T(x) = \varepsilon_1(x) \left( \frac{1}{x} - (2k-1) \right), \label{1.5}
\end{equation*}
where $\varepsilon_1(x) = 1$ if $x \in \left( \frac{1}{2k}, \frac{1}{2k-1} \right]$ and
$\varepsilon_1(x) = -1$ if $x \in \left( \frac{1}{2k-1}, \frac{1}{2k-2} \right]$.
We obtain
\begin{equation*}
x = \displaystyle \frac{1}{2k-1 + \varepsilon_1(x) T(x)} \label{1.6}
\end{equation*}
and therefore the map $T$ generates the continued fraction
\begin{equation} \label{1.7}
x = \displaystyle \frac{1}{a_1 + \displaystyle \frac{\varepsilon_1}{a_2 + \displaystyle \frac{\varepsilon_2}{a_3 + \ddots}}}
=: [1/a_1, \varepsilon_1/a_2, \varepsilon_2/a_3, \ldots],
\end{equation}
where
\begin{equation} \label{1.8}
\begin{split}
   a_n=a_n(x) = a_1\left( T^{n-1}(x) \right), \, \varepsilon_n=\varepsilon_n(x) = \varepsilon_1\left( T^{n-1}(x) \right), n \geq 1, \\
   \varepsilon_n  \in \{-1, +1\},\, a_n \geq 1, \, a_n \equiv 1(\mathrm{mod}\, 2)\mbox{ and } a_n +\varepsilon_n >1, \, n \geq 1.
\end{split}
\end{equation}
On the OCF expansion the iterates of the map $T$ act as a shift map by
\begin{equation*} \label{1.9}
T^n ([1/a_1, \varepsilon_1/a_2, \varepsilon_2/a_3, \ldots]) = [1/a_{n+1}, \varepsilon_{n+1}/a_{n+2}, \varepsilon_{n+2}/a_{n+3}, \ldots].
\end{equation*}
Let us denote
\begin{equation} \label{1.10}
\begin{split}
   r_n & = r_n(x) = a_n(x) + \varepsilon_n(x) T^n(x) \\
       & = a_n + [\varepsilon_n/a_{n+1}, \varepsilon_{n+1}/a_{n+2}, \varepsilon_{n+2}/a_{n+3}, \ldots], \quad  n\geq 1
\end{split}
\end{equation}
which yields
\begin{equation} \label{1.11}
r_n = a_n + \frac{\varepsilon_n}{r_{n+1}}, \, n\geq 1.
\end{equation}

The rational approximants to $x$ arise in a manner similar to that in the case of other continued fraction algorithms.
However, the OCF case is the most intricate, because the sequence of denominators of successive convergents in OCF$(x)$ is not necessarily increasing as in the regular continued fractions (RCF) or in the continued fractions with even partial quotients (ECF) cases.

Let us define
\begin{eqnarray}
  p_{-1} &=& 1, \, p_{0} =0, \, p_n=a_np_{n-1} + \varepsilon_{n-1}p_{n-2}, \label{1.12} \\
  q_{-1} &=& 0, \, q_{0} =1, \, q_n=a_nq_{n-1} + \varepsilon_{n-1}q_{n-2} \label{1.13}
\end{eqnarray}
for $n \geq 1$.
The sequence of rationals $\left\{ \displaystyle \frac{p_n}{q_n} \right\}$, $n \geq 1$ are the convergents to $x$ in $[0, 1]$.

The following elementary fundamental relations are satisfied
\begin{eqnarray}
  && p_{n-1} q_n - p_n q_{n-1} = (-1)^k \varepsilon_0 \varepsilon_1 \ldots \varepsilon_{n-1} =: \delta_n, \label{1.14}\\
  && \frac{p_{n-1}}{q_{n-1}} - \frac{p_{n}}{q_{n}}  = \frac{\delta_n}{q_{n-1}q_n}, \, n \geq 0, \label{1.15} \\
  && x = \frac{p_n+p_{n-1}\varepsilon_n t_n}{q_n+q_{n-1}\varepsilon_n t_n}, \, n \geq 0, \label{1.16}
\end{eqnarray}
where $t_n = T^n(x)$.
Equation (\ref{1.16}) is equivalent to
\begin{equation}\label{1.17}
  \varepsilon_n t_n = \frac{q_nx -p_n}{-q_{n-1}x + p_{n-1}}, \, n \geq 0.
\end{equation}
Upon (\ref{1.17}) we infer that for any irrational number $x \in [0, 1]$
\begin{equation*}\label{1.18}
 0 < \left| \frac{q_nx-p_n}{-q_{n-1}x+ p_{n-1}} \right| <1, \,  n \geq 0.
\end{equation*}
Let us consider
\begin{equation}\label{1.19}
s_0(x) = 0, \, s_1(x) = \frac{1}{a_1(x)}, \, s_n(x) = \frac{1}{a_n(x)+\varepsilon_{n-1}(x)s_{n-1}(x)}, \, n\geq 2
\end{equation}
which yields
\begin{equation}\label{1.20}
  s_n = [1/a_n, \varepsilon_{n-1}/a_{n-1}, \varepsilon_{n-2}/a_{n-2}, \ldots, \varepsilon_1/a_1].
\end{equation}
Obviously, $s_n = \displaystyle \frac{q_{n-1}}{q_n}$, $n \geq 0$.

The golden ratios $G = \frac{\sqrt{5}+1}{2}=1.6180\dots$ and $g = \frac{\sqrt{5}-1}{2}=0.6180\dots$ will be used often.
Without further mention we shall frequently use identities like
\begin{equation*}
g+1=G, \, g^2=1-g, \, G^2=G+1, \, gG=1, \, g+2 = G^2.
\end{equation*}
Denominators of successive convergents for OCF$(x)$ satisfy
\begin{equation} \nonumber
  \begin{split}
     \frac{q_n}{q_{n-1}} &= a_n + \varepsilon_{n-1} [1/a_{n-1}, \varepsilon_{n-2}/a_{n-2}, \ldots, \varepsilon_1/a_1] \geq  \\
       &\geq a_n - [1/3, -1/3, \ldots, -1/3] \\
       & > a_n - [1/3, -1/3, -1/3, \ldots] = a_n -1 + \frac{1}{G} = a_n -2 +G.
  \end{split}
\end{equation}
Also, one has
\begin{equation*}\label{1.21}
 \frac{q_n}{q_{n-1}} = a_n + \frac{\varepsilon_{n-1}}{\frac{q_{n-1}}{q_{n-2}}} < a_n + \frac{\varepsilon_{n-1}}{a_{n-1}-2+G} \leq a_n + \frac{1}{G-1} = a_n +G.
\end{equation*}
Whatever $n \geq 1$ we have $a_n \geq 3$ if and only if $0 < s_n < g^2$ and $a_n =1$ if and only if $g^2 < s_n < G$.

Studied by Rieger \cite{Rieger-1981} and Schweiger \cite{Schweiger-1982, Schweiger-1984} in the early time, the OCF continue to raise the interest in many ways. Any continued fraction algorithm on $[0, 1]$ generates a natural filtration $\left\{\mathcal{Y}_n \right\}$ of $\mathbb{Q} \cap [0, 1]$, obtained by taking into account the sum of the partial quotients of the rationals.
One can consider the simple non-decreasing functions
\[
Q_n : [0, 1] \rightarrow [0, 1], \quad Q_n(x) = \frac{\left| \left\{ y \in \mathcal{Y}_n : y < x \right\} \right|}{\left| \mathcal{Y}_n \right| - 1}. 
\]
For the regular continued fraction, $\mathcal{Y}_n$ is the set of rationals with sum of partial quotients at most $n$.
The limit $Q(x) := \displaystyle \lim_{n \rightarrow \infty} Q_n(x)$ provides an analogue of the Minkowski question mark function.
The resulting map $Q_O$ in the situation of continued fractions with odd partial quotients has been introduced and investigated by Zhabitskaya \cite{Zhabitskaya-2012}.
Actually, the analogue of Minkowski's question mark function related to continued fractions with odd partial quotients $Q_O(x)$ coincides with her $F^0(x)$.
Following Zhabitskaya's work \cite{Zhabitskaya-2012}, Boca and Linden \cite{Boca&L-2018} proved that the function $Q_O$ is H\"older continuous with best exponent $\frac{\log \lambda}{2 \log G} \approx 0.63317$, where $\lambda \approx 1.83929$ denotes the unique real root of the equation $x^3-x^2-x-1=0$.
They also proved that the map $Q_O$ linearizes the odd Gauss and the odd Farey maps.

Boca and Merriman \cite{Boca&M-2018} described coding of geodesics on the modular surface $\mathcal{M}_O$ connected to the dynamics of odd continued fractions.

Recently, Boca and Merriman \cite{Boca&M-2019} studied an analogue of Nakada's $\alpha$-continued fraction transformation in the setting of continued fractions with odd partial quotients and described the natural extension of this transformation.

The purpose of this paper is to develop a different approach from \cite{Popescu-1997b} and \cite{Sebe} to study the asymptotic behavior of the distribution functions in the OCF case.
Based on the finite $T$-invariant measure on the $\sigma$-algebra $\mathcal{B}_{[0,1]}$ of all Borel subsets of $[0,1]$
\begin{equation}\label{1.22}
  \rho (A) = \frac{1}{3\log G} \int_{A} \left( \frac{1}{x+G-1} - \frac{1}{x-G-1} \right)\mathrm{d}x, \quad A \in \mathcal{B}_{[0,1]}
\end{equation}
introduced by Schweiger \cite{Schweiger-1982}, Kalpazidou \cite{Kalpazidou-1986} investigated the ergodic behavior of a certain homogeneous random system with complete connections (RSCC). In fact, it was proved that this RSCC is uniformly ergodic and its associated transition operator under the invariant measure $\rho$ is regular with respect to the Banach space of Lipschitz functions.
These results allowed to find the limit $ \displaystyle \lim_{ n \to \infty} \mu (r_n >t) = \ell $
for a given non-atomic measure $\mu$ on $\mathcal{B}_{[0,1]}$ and to estimate the error $\mu (r_n >t) - \ell$.

On the other hand, Popescu \cite{Popescu-1997a, Popescu-1997b} studied the $[0, G]$-valued Markov chain $\left\{ s_n \right\}$, ${n\geq 1}$ and its associated RSCC and solved a variant of the Gauss-Kuzmin problem.

Sebe \cite{Sebe} solved a Gauss-Kuzmin-type problem for the OCF expansion by considering the transition operator of the Markov chain
$\left\{ s_n \right\}$, ${n\geq 1}$ as an operator on a certain Banach space of complex-valued functions of bounded variation.
It was also obtained an improvement of the result given in \cite{Popescu-1997b} concerning the convergence rate.
It should be said that Wirsing's method \cite{Wirsing-1974} cannot be applied in this case because the corresponding transition operator is not positive.

The paper is organized as follows. In Section 2 we introduce the definitions and present preliminary results. We also make a brief survey of the method used in \cite{Sebe}.

In Section 3, to continue our investigation on the asymptotic behavior of the distribution functions of the map $T$, we use an approach in the spirit of Sz\"usz \cite{Szusz-1961}. We mention that using the method of Sz\"usz to prove a Gauss-Kuzmin-type theorem for OCF expansions, we obtain a very good estimate of the convergence rate involved.

\section{Preliminary results}
\label{Sec:2}
\subsection{The random system with complete connections associated with the sequence $\left\{ s_n \right\}$}
In the sequel we present some investigations on the sequence $\left\{ s_n \right\}$, ${n\geq 1}$ studied in detail in \cite{Popescu-1997a}.

Denoting by $\displaystyle E_{i_1 i_2 \ldots i_n}^{1 j_1 \ldots j_{n-1}}$ the set of irrational numbers
$$x=\left[ 1/a_1(x), \varepsilon_1/a_2(x), \varepsilon_2/a_3(x), \ldots \right] \in [0, 1]$$
for which $a_{\ell} (x) = i_{\ell}$, $1 \leq \ell \leq n$, $\varepsilon_k (x) = j_{k}$, $1 \leq k \leq n-1$, with $j_k = \pm 1$, $i_k +j_k >1$, $1 \leq k \leq n-1$, $i_{\ell} \geq 1$, $i_{\ell} \equiv 1(\mathrm{mod}\, 2)$, $1 \leq \ell \leq n$, $n \geq 2$, we have
\begin{equation}\label{2.1}
\lambda \left( r_{n+1} > t, \, \varepsilon_n = e \left| E_{i_1 i_2 \ldots i_n}^{1 j_1 \ldots j_{n-1}} \right. \right) =
\left\{
\begin{array}{lll}
\displaystyle \frac{1-s^2_n}{2(t+e s_n)}, & i_n \neq 1, \, e= \pm 1, \\
0,                                        & i_n =1, \, e=-1,  \\
\displaystyle \frac{1+s_n}{t+s_n},        & i_n =1, \, e=1,
\end{array} \right.
\end{equation}
where $t \geq 1$, $n \geq 1$, and $\lambda$ is the Lebesgue measure (see \cite{Popescu-1997a}).
Also, since $r_1 (x) = \displaystyle \frac{1}{x}$, we have
\[
\lambda \left( r_1(x) > t \right) = \lambda \left( \left[0, \frac{1}{t} \right] \right) = \frac{1}{t}.
\]
As is well-known ( see \cite{Popescu-1997a}) the sequence $\left\{ s_n \right\}$, ${n\geq 1}$ is a $W$-valued Markov chain with transition probability function $Q$ defined as
\[
Q(w, B) = \sum_{\left\{ (e,i) : u(w(e,i)) \in B \right\}} P(w(e,i)), \quad w \in W, \, B \in \mathcal{W},
\]
where $W = [0, G]$, $\mathcal{W} = \mathcal{B}_{[0, G]}$= the collection of all Borel subsets of $W$,
$X = \{-1,1\}\times\{1,3,5,\ldots \}$, $\mathcal{X}$=the collection of all subsets of $X$,
\begin{equation*}\label{2.2}
u(w, (-1,i)) =
\left\{
\begin{array}{lll}
\displaystyle \frac{1}{i-w},     & w \in \left[ 0, g^2 \right)=:W_1, \\
\displaystyle \frac{1}{i-g^2},   & w \in \left[ g^2, G \right] =: W_2,
\end{array} \right.,
\quad u(w, (1,i)) = \frac{1}{i+w}, \quad w \in W,
\end{equation*}
\begin{equation}\label{2.4}
P(w, (e,i)) =
\left\{
\begin{array}{lll}
\displaystyle \frac{(1-w^2)(2-\delta(i,1))}{2(i-1+\delta(i,1)+ew)(i+1+ew)},  & w \in W_1, \\
\\
\displaystyle \frac{(1+w)(2-\delta(i,1))\delta(e,1)}{(i-1+\delta(i,1)+w)(i+1+w)},   & w \in W_2,
\end{array} \right.
\end{equation}
for all $i \geq 1$, $i\equiv 1(\mathrm{mod}\, 2)$ and $ e = \pm 1$, where $\delta$ is the Kronecker's symbol.

According to the general theory \cite{IG-2009}
\begin{equation} \label{2.5}
\left((W, \mathcal{W}), (X, \mathcal{X}), u, P\right)
\end{equation}
is the random system with complete connections (RSCC) associated with the sequence $\left\{ s_n \right\}$, ${n\geq 1}$.
Let $Q^n$, $n \geq 1$, be the $n$-step transition probability function associated with $Q$.
The stationary probability $Q^{\infty}$ for $\left\{ s_n \right\}$, ${n\geq 1}$ is given by
\begin{equation} \label{2.6}
Q^{\infty}(B) = \int_{B} \mathrm{d} \xi(w), \quad B \in \mathcal{W},
\end{equation}
where
\begin{equation} \label{2.7}
\xi(w) =
\left\{
\begin{array}{lll}
\displaystyle \frac{1}{3\log G} \log \frac{1+w}{1-w},  & w \in W_1, \\
\\
\displaystyle \frac{1}{3\log G} \log \frac{1+w}{1-g^2} ,  & w \in W_2.
\end{array} \right.
\end{equation}
Therefore, for any $B \in \mathcal{B}_{[0,G]}$ we have
\begin{equation}\label{2.8}
\int_{W} Q^{\infty} (\mathrm{d}w)Q(w,B) = Q^{\infty}(B).
\end{equation}

\subsection{An operatorial treatment}

Let us consider the transition operator $U$ associated with RSCC (\ref{2.5}) which is defined as
\begin{equation*}\label{3.1}
U f(w) = \sum_{(e,i) \in X} P(w, (e,i)) f(u(w, (e,i))), \quad w \in W,
\end{equation*}
for any $f \in B(W)$ (= the Banach space of bounded measurable complex-valued functions $f$ on $W$ under the supremum norm $|f| = \displaystyle \sup_{w \in W} |f(w)|$).
Note that $U$ is also the transition operator of the Markov chain $\left\{ s_n \right\}$, ${n \geq 1}$ and we have
\begin{equation*}\label{3.2}
U f(w) = \int_{W} Q(w, \mathrm{d} w')f(w')
\end{equation*}
which implies that
\begin{equation*}\label{3.3}
U^n f(w) = \int_{W} Q^n(w, \mathrm{d} w')f(w'), \quad w \in W, \, n \geq 1.
\end{equation*}

The basic ideea is to consider $U$ as an operator on $BV(W)$(=the Banach space of all complex-valued functions $f$ of bounded variation on $W$ under the norm $\|f\|_{\mathrm{v}} = \mathrm{var } f + |f|$). Remember that the variation $\mathrm{var }_A f$ over $A \subset W$ of $f \in B(W)$ is defined as
\[
\sup \sum_{i=1}^{k-1} \left| f(t_i) - f(t_{i+1}) \right|
\]
the supremum being taken over all $t_1 < \ldots < t_k \in A$ and $k \geq 2$.
We write simply $\mathrm{var } f$ for $\mathrm{var }_W f$ and, if $\mathrm{var } f < \infty$, then $f$ is called a function of bounded variation.

We recall two elementary results obtained in \cite{Sebe}.
\begin{proposition} \label{Prop1}
For any $f \in BV(W)$ we have
\begin{equation}\label{3.4}
\mathrm{var }\, Uf \leq \theta_1 \mathrm{var } f + \theta_2 |f|,
\end{equation}
where $\theta_1$ and $\theta_2$ are positive constants such that $\theta_1=0.4270508\ldots$ and $\theta_2 \leq 0.396312\ldots$
\end{proposition}
\begin{corollary} \label{Cor3.2}
  There exists a positive constant $\theta = \theta_1 + \theta_2 \leq 0.8233628\ldots$ such that for all $n \in \mathbb{N}$ and $f \in BV(W)$ we have
  \begin{eqnarray}
      \mathrm{var }\, U^n f &\leq& \theta^n \cdot \mathrm{var }\,f, \label{3.017} \\
      \left| U^n f - U^{\infty}f  \right| &\leq& \theta^n  \cdot \mathrm{var }\,f, \label{3.18}
  \end{eqnarray}\label{3.17}
where $U^{\infty} f = \displaystyle \int_{W} f(w) Q^{\infty} (\mathrm{d}w)$.
\end{corollary}

Let us define two functions
\begin{equation}\label{3.01}
  F_n(x,e) : = \lambda \left( \left. r_{n+k+1} > \frac{1}{x}, \varepsilon_{n+k} = e \right| E_{i_1i_2\ldots i_k}^{1 j_1 \ldots j_{k-1}} \right)
\end{equation}
and
\begin{equation}\label{3.02}
  F(x,e) := \displaystyle \int_{W_1} \frac{(1-y^2)x}{2(1+yx)}Q^{\infty}(\mathrm{d}y)+
  \int_{W_2} \frac{(1+y)\delta(e,1)x}{ 1+yx }Q^{\infty}(\mathrm{d}y)
\end{equation}
for $x \in [0,1]$, $e = \pm1$, $n > 0$ and $k \geq 1$.

The Corollary \ref{Cor3.2} allows us to solve a Gauss-Kuzmin-type problem, namely to obtain the asymptotic behavior as $n \to \infty$ of the distribution function $F_n$ and estimate the convergence rate.
In fact, we show that for all $n \geq 1$, $x \in [0,1]$, $e = \pm1$, we have
\begin{equation*} \label{3.03}
  \left| F_n(x,e) - F(x,e) \right| \leq \theta^n.
\end{equation*}
Our convergence rate, $\mathcal{O}\left(\theta^n\right)$, with $\theta\leq 0.8233628\ldots < 0.854120\ldots$, is better than the one obtained in \cite{Popescu-1997b}.
To proceed, by (\ref{2.1}) we have
\begin{equation}\label{3.25}
\begin{split}
   F_0 (x,e) &= \lambda \left( \left. r_{k+1} > \frac{1}{x}, \varepsilon_k = e \right| E_{i_1 i_2 \ldots i_k}^{1 j_1 \ldots j_{k-1}} \right) \\
             &=
 \left\{
        \begin{array}{lll}
            \displaystyle \frac{(1-s^2_k)x}{2(1+e s_k x)}, & s_k \in W_1, e= \pm 1, \\
            0,                                             & s_k \in W_2, e= -1, \\
            \displaystyle \frac{(1+s_k)x}{1+s_k x},        & s_k \in W_2, e= 1.
        \end{array}
\right.
\end{split}
\end{equation}
Note that we also have
\begin{equation}\label{3.26}
F_0 (x,e) =
\left\{
        \begin{array}{lll}
           \displaystyle\int_{W_1} \frac{(1-y^2)x}{2(1+yx)} \mathrm{d}G_0(y) + \int_{W_2} \frac{(1+y)x}{1+yx} \mathrm{d}G_0(y) , &e=1, \\
           \\
           \displaystyle\int_{W_1} \frac{(1-y^2)x}{2(1-yx)} \mathrm{d}G_0(y) , &e=-1,
        \end{array}
\right.
\end{equation}
with
\begin{equation}\label{3.27}
G_0(y) =
\left\{
        \begin{array}{lll}
           0, & y \leq s_k, \\
           1, & y > s_k,
        \end{array}
\right. \, y \in W.
\end{equation}
Now, for any $n \geq 1$, by (\ref{3.25}) we have
\begin{equation*}\label{3.28}
  \begin{split}
     F_n(x,e) &= \sum \lambda \left( r_{n+k+1} > \frac{1}{x}, \varepsilon_{n+k}=e, \varepsilon_{n+k-1}=j_n, \right.\\
     &\left. \left. \qquad \qquad a_{n+k} = i_n, \ldots, \varepsilon_{k}=j_1, a_{k+1}=i_1 \right| E_{i_1 i_2 \ldots i_k}^{1 j_1 \ldots j_{k-1}} \right) \\
     &= \sum  \lambda \left( \left. r_{n+k+1} > \frac{1}{x}, \varepsilon_{n+k}=e \right| E_{i_1 \ldots i_k, i'_1 \ldots i'_n}^{1 \ldots j_{k-1}, j'_1 \ldots j'_n}  \right) \\
     &\qquad \, \,   \times \lambda \left( \left. \varepsilon_{n+k-1}=j'_n, a_{n+k} = i'_n, \ldots, \varepsilon_{k}=j'_1, a_{k+1}=i'_1 \right| E_{i_1 \ldots i_k}^{1 \ldots j_{k-1}} \right) \\
     &=
      \left\{
        \begin{array}{lll}
           \displaystyle\int_{W_1} \frac{(1-y^2)x}{2(1+yx)} \mathrm{d}G_n(y) + \int_{W_2} \frac{(1+y)x}{1+yx} \mathrm{d}G_n(y) , &e=1, \\
           \\
           \displaystyle\int_{W_1} \frac{(1-y^2)x}{2(1-yx)} \mathrm{d}G_n(y) , &e=-1,
        \end{array}
\right.
  \end{split}
\end{equation*}
where the sums are taken over all $i'_1, \ldots, i'_n \geq 1$ and $j'_{\ell} = \pm 1$, $1 \leq \ell \leq n$, for which
$i'_{\ell} \equiv 1(\mathrm{mod}\, 2)$, $i'_k+ j'_k >1$, $1 \leq \ell \leq n$, $1 \leq k \leq n-1$, and
\begin{equation*}\label{3.29}
  G_n(y) = G_n \left( y, E_{i_1 i_2 \ldots i_k}^{1 j_1 \ldots j_{k-1}} \right) = Q^n \left( s_k, [0,y) \right), \, y \in W, \, n \geq 1.
\end{equation*}

Let us define $G^{\infty} (y) = Q^{\infty} ([0,y))$, $y \in W$.
\begin{theorem} \label{Th.G}
For all $n \geq 1$, $x \in [0, 1]$, $e= \pm1$, $\theta \leq 0.8233628\ldots$ and $y \in W$ we have
\begin{eqnarray}
  &&\left| F_n(x,e) - F(x,e) \right| \leq \theta^n, \label{3.30} \\
  &&\left| G_n(y) - G^{\infty}(y)\right| \leq \theta^n. \label{3.031}
\end{eqnarray}
\end{theorem}
\begin{proof}
\, First, we have
\begin{equation*}\label{3.31}
  \begin{split}
      & \left| F_n(x,1) - F(x,1) \right| \\
      &\qquad = \left| \displaystyle\int_{W_1} \frac{(1-y^2)x}{2(1+yx)} \mathrm{d} \left(G_n(y) - G^{\infty}(y)\right) + \int_{W_2} \frac{(1+y)x}{1+yx} \mathrm{d}\left(G_n(y) - G^{\infty}(y)\right) \right|  \\
       &\qquad  \leq \displaystyle\int_{W_1} \left| G_n(y) - G^{\infty}(y) \right| \frac{x^2+2xy+x^2y^2}{2(1+yx)^2} \mathrm{d}y  \\
       &\qquad + \displaystyle\int_{W_2} \left| G_n(y) - G^{\infty}(y) \right| \frac{x(1-x)}{(1+yx)^2} \mathrm{d}y + \left| G_n(g^2) - G^{\infty}(g^2) \right| \frac{5g^2x}{2(1+g^2x)}.
  \end{split}
\end{equation*}
Also,
\begin{equation}\label{3.32}
      \left| F_n(x,-1) - F(x,-1) \right| = \left| \displaystyle\int_{W_1} \frac{(1-y^2)x}{2(1-yx)} \mathrm{d}\left(G_n(y) - G^{\infty}(y)\right)  \right|.
\end{equation}
If $0 < x \leq \displaystyle \frac{2g^2}{g^4+1}$, then (\ref{3.32}) becomes
\begin{equation}\label{3.33}
\begin{split}
      &\left| F_n(x,-1) - F(x,-1) \right| \leq \int_{0}^{\frac{1-\sqrt{1-x^2}}{x}}  \left| G_n(y) - G^{\infty}(y) \right| \frac{x(xy^2+x-2y)}{2\left(1-yx^2\right)}\mathrm{d}y \\
      &\qquad \qquad \qquad + \int_{\frac{1-\sqrt{1-x^2}}{x}}^{g^2} \left| G_n(y) - G^{\infty}(y) \right| \frac{x(-xy^2-x+2y)}{2(1-yx^2)}\mathrm{d}y \\
      & \qquad \qquad \qquad + \left| G_n(g^2) - G^{\infty}(g^2) \right| \frac{(1-g^4)x}{2(1-g^2x)}.
\end{split}
\end{equation}
And if $\displaystyle \frac{2g^2}{g^4+1} < x \leq 1$, then (\ref{3.32}) becomes
\begin{equation}\label{3.34}
\begin{split}
\left| F_n(x,-1) - F(x,-1) \right| &\leq \displaystyle\int_{W_1} \left| G_n(y) - G^{\infty}(y) \right| \frac{x(xy^2+x-2y)}{2(1-yx)^2} \mathrm{d}y \\
                                     &+ \left| G_n(g^2) - G^{\infty}(g^2) \right| \frac{(1-g^4)x}{2(1-g^2x)}.
\end{split}
\end{equation}
Now, for all $y \in W$ we have
\begin{equation}\label{3.35}
  \left| G_n(y) - G^{\infty}(y) \right| = \left| Q^n (s_k, [0,y)) - Q^{\infty} ([0,y))\right|=\left|U^n f_y(s_k) - U^{\infty} f_y\right|,
\end{equation}
where $U$ is the transition operator of the Markov chain $\left\{s_n\right\}$, $n \geq 1$,
$\displaystyle U^{\infty} f = \int_{W} f(w)Q^{\infty}(\mathrm{d}y)$, and $f_y$ is a function defined on $W$ as
\begin{equation}\label{3.36}
  f_y(w)=
\left\{
        \begin{array}{lll}
            1, & 0\leq w \leq y, \\
            0, & y< w \leq G.
        \end{array}
\right.
\end{equation}
Hence, by (\ref{3.18}), for all $y \in W$ and $n >0$ we obtain
\begin{equation}\label{3.37}
  \left|U^n f_y(s_k) - U^{\infty} f_y\right| \leq \theta^n \, \mathrm{var }\, f_y = \theta^n.
\end{equation}
It follows that for all $y \in W$ and $n >0$ we have (\ref{3.031}).
%
%
Now, we can obtain a good estimate of $\left|F_n(x,e) - F(x,e)\right|$, namely
\begin{equation}\label{3.39}
\begin{split}
   &\left|F_n(x,e) - F(x,e)\right|  \\
   &\leq
  \left\{
        \begin{array}{lll}
            \displaystyle \theta^n \sup \left(\frac{xy(x^2+2x+7g+3)}{2(1+g^2x)(1+Gx)}\right), & e=1, \\
            \displaystyle \theta^n \sup \left(\frac{2( 1-\sqrt{1-x^2})}{x} - \frac{x}{2}\right), & e=-1 \mbox{ and } 0 < x \leq \displaystyle \frac{2g^2}{g^4+1}, \\
            \displaystyle \theta^n \sup \left(\frac{x (g^2x+1-2g^4)}{2(1-g^2x)}\right), & e=-1 \mbox{ and } \displaystyle \frac{2g^2}{g^4+1} < x \leq 1.
        \end{array}
\right.
\end{split}
\end{equation}
As is easy to see, the supremum does not exceed $1$, so that we get (\ref{3.30}).
%
%
\end{proof}

\section{A new solution of the Gauss-Kuzmin-type problem}
Let $\mu$ be a non-atomic probability measure on $\mathcal{B}_{[0,1]}$ and define:
\begin{eqnarray}
H_{0} (x) &:=& \mu ([0, x]), \ x \in [0, 1], \label{4.1} \\
H_{n} (x) &:=& \mu (r^{-1}_{n+1} < x), \ x \in [0, 1], \ n \geq 1 \label{4.2}
\end{eqnarray}
where $r_n$ is as in (\ref{1.10}). Then the following holds.
\begin{theorem}  \label{Th.GKL}
Let $r_n$ and $H_{n}$ be as in $(\ref{1.10})$ and $(\ref{4.2})$.
Then there exists a constant $0 < \eta < 1$ such that $H_n$ can be written as
\begin{equation}
H_{n} (x) = \frac{1}{3 \log G} \log \frac{(G+1)(G-1+x)}{(G-1)(G+1-x)} + \mathcal{O}(\eta^n) \label{4.3}
\end{equation}
uniformly with respect to $x \in [0, 1]$.
\end{theorem}

To prove Theorem \ref{Th.GKL} we need the following results.
\begin{lemma} \label{GK.eq.}
For functions $\{H_{n}\}$ in $(\ref{4.2})$, the following Gauss-Kuzmin-type equation holds:
\begin{equation}
H_{n+1} (x) = \sum_{(i,\varepsilon)} \varepsilon \left(H_{n}\left(\frac{1}{i}\right) - H_{n}\left(\frac{1}{i+\varepsilon x}\right)\right) \label{4.4}
\end{equation}
for $x \in [0,1]$ and $n >0$. Here $(i,\varepsilon)$ denotes that $i \equiv 1 (\mathrm{mod }\, 2)$, $| \varepsilon | = 1$ and $i+\varepsilon>1$.
\end{lemma}
\begin{proof}
\, From $(\ref{1.11})$ and $(\ref{4.2})$ we have
\begin{equation*}
  \begin{split}
      H_{n+1}(x) =& \mu \left( r^{-1}_{n+2} < x, \, \varepsilon_{n+1} = 1 \right) + \mu \left( r^{-1}_{n+2} < x, \, \varepsilon_{n+1} = -1 \right)\\
       =& \sum_{i \equiv 1 (\mathrm{mod }\, 2)} \mu \left( \frac{1}{i+x} < r^{-1}_{n+1} < \frac{1}{i} \right) \\
       &+ \sum_{i \equiv 1 (\mathrm{mod }\, 2),i \neq 1} \mu \left( \frac{1}{i} < r^{-1}_{n+1} < \frac{1}{i-x} \right) \\
       =& \sum_{(i,\varepsilon)} \varepsilon \left(H_{n}\left(\frac{1}{i}\right) - H_{n}\left(\frac{1}{i+\varepsilon x}\right)\right).
  \end{split}
\end{equation*}
\end{proof}
\begin{remark}
Assume that for some $p >0$, the derivative $H'_{p}$ exists everywhere in $[0, 1]$ and is bounded.
Then it is easy to see by induction that $H'_{p+n}$ exists and is bounded for all $n \geq 1$.
This allows us to differentiate $(\ref{4.4})$ term by term, obtaining
\begin{equation}
H'_{n+1}(x) = \sum_{(i, \varepsilon)} \frac{1}{(i+\varepsilon x)^2}H'_{n}\left(\frac{1}{i+\varepsilon x}\right). \label{4.5}
\end{equation}
\end{remark}
We introduce functions $\{ h_{n} \}$ as follows:
\begin{equation}
h_{n}(x) := \frac{H'_{n}(x)}{{(x+G-1)^{-1}} - {(x-G-1)}^{-1}}, \quad x \in [0, 1], \ n >0. \label{4.6}
\end{equation}
Then (\ref{4.5}) is
\begin{equation}\label{4.7}
h_{n+1} (x) = \left( G^2 - (1-x^2) \right) \sum_{(i, \varepsilon)} V(x, (i,\varepsilon)) h_n\left(\frac{1}{i+\varepsilon x}\right),
\end{equation}
where
\begin{equation}\label{4.8}
V(x, (i,\varepsilon)) = \frac{1}{\left( (G-1)(i+\varepsilon x) +1 \right)\left( (G+1)(i+\varepsilon x) - 1 \right)}.
\end{equation}
\begin{lemma} \label{lem.4.4}
For $\left\{ h_n\right\}$ in (\ref{4.6}), define $M_n:=\displaystyle \max_{x \in [0, 1]} \left| h'_n(x) \right|$.
Then
\begin{equation}\label{4.9}
  M_{n+1} \leq \eta \cdot M_n,
\end{equation}
where
\begin{equation}\label{4.10}
  \eta := 4 g \sum_{i=1,3,\ldots} \frac{1}{(G+i)i(i+2)}.
\end{equation}
\end{lemma}
\begin{proof}
\, Before we derive $(\ref{4.7})$, we bring it to a convenient form. First,
\begin{equation}\label{4.11}
\begin{split}
&h_{n+1} (x) = \left( G^2 - (1-x^2) \right) \\
&\times \left( \sum_{i=1,3,\ldots} V(x, (i,1)) h_n\left(\frac{1}{i+x}\right) + \sum_{i=3,5,\ldots} V(x, (i,-1)) h_n\left(\frac{1}{i-x}\right) \right).
\end{split}
\end{equation}
Since $\displaystyle\frac{1}{G-1}=G$ \, and \, $\displaystyle\frac{1}{G+1} = 2-G$, then
\begin{equation*}
\begin{split}
  V(x,(i,1)) &= \frac{1}{G^2-1} \frac{1}{(i+x+G)(i+x+G-2)} \\
             &= \frac{1}{2\left( G^2-1 \right)} \left( \frac{1}{i-2+x+G}-\frac{1}{i+x+G} \right).
\end{split}
\end{equation*}
Similarly, we get
\begin{equation*}\label{4.13}
  V(x,(i,-1)) = \frac{1}{2\left( G^2-1 \right)} \left( \frac{1}{i-2-x+G} - \frac{1}{i-x+G} \right).
\end{equation*}
Therefore, $(\ref{4.11})$ becomes
\begin{equation*}\label{4.14}
\begin{split}
h_{n+1} (x) &= \frac{\left( G^2 - (1-x^2) \right)}{2\left( G^2-1 \right)} \left( \sum_{i=1,3,\ldots} \left( \frac{1}{i-2+x+G} - \frac{1}{i+x+G} \right)h_n\left(\frac{1}{i+x}\right) \right. \\
 & + \left. \sum_{i=3,5,\ldots} \left( \frac{1}{i-2-x+G} - \frac{1}{i-x+G} \right) h_n\left(\frac{1}{i-x}\right) \right).
\end{split}
\end{equation*}
Now, by calculus we have
\begin{equation}\label{4.15}
h'_{n+1} (x) = \frac{g}{2} \left\{ 2(1-x) \left( S_{1} + S_{2}  \right) + \left( G^2 -(1-x)^2 \right) \left(  S_{3}+  S_{4} -  S_{5} +  S_{6} \right)\right\}
\end{equation}
where
\begin{eqnarray}
S_{1} &:=& \sum_{i=1,3,\ldots} \left( \frac{1}{i-2+x+G} - \frac{1}{i+x+G} \right) h_n\left(\frac{1}{i+x}\right), \label{4.16} \\
S_{2} &:=& \sum_{i=3,5,\ldots} \left( \frac{1}{i-2-x+G} - \frac{1}{i-x+G} \right) h_n\left(\frac{1}{i-x}\right), \label{4.17} \\
S_{3} &:=& \sum_{i=1,3,\ldots} \left( \frac{1}{(i+x+G)^2} - \frac{1}{(i-2+x+G)^2} \right) h_n\left(\frac{1}{i+x}\right), \label{4.18} \\
S_{4} &:=& \sum_{i=3,5,\ldots} \left( \frac{1}{(i-2-x+G)^2} - \frac{1}{(i-x+G)^2} \right) h_n\left(\frac{1}{i-x}\right), \label{4.19} \\
S_{5} &:=& \sum_{i=1,3,\ldots} \left( \frac{1}{i-2+x+G} - \frac{1}{i+x+G} \right) h'_n\left(\frac{1}{i+x}\right)\frac{1}{(i+x)^2}, \label{4.20} \\
S_{6} &:=& \sum_{i=3,5,\ldots} \left( \frac{1}{i-2-x+G} - \frac{1}{i-x+G} \right) h'_n\left(\frac{1}{i-x}\right)\frac{1}{(i-x)^2}. \label{4.21}
\end{eqnarray}
But,
\begin{eqnarray}
S_{1} &=& \frac{h_n \left( \frac{1}{x+1} \right)}{x+G-1}  - \sum_{i=1,3,\ldots} \frac{2}{(x+G+i)(x+i)(x+i+2)} h'_n(\alpha_i), \label{4.22} \\
S_{2} &=& \frac{h_n \left( \frac{1}{3-x} \right)}{-x+G+1}  - \sum_{i=3,5,\ldots} \frac{2}{(i-x+G)(i+2-x)(i-x)} h'_n(\beta_i), \label{4.23} \\
S_{3} &:=& \frac{-h_n \left( \frac{1}{x+1} \right)}{(x+G-1)^2}  + \sum_{i=1,3,\ldots} \frac{2}{(x+G+i)^2(x+i)(x+i+2)} h'_n(\alpha_i), \label{4.24} \\
S_{4} &=& \frac{h_n \left( \frac{1}{3-x} \right)}{(-x+G+1)^2}  - \sum_{i=3,5,\ldots} \frac{2}{(i-x+G)^2(i+2-x)(i-x)} h'_n(\beta_i), \label{4.25}
\end{eqnarray}
where $\displaystyle \frac{1}{x+i+2} < \alpha_i < \displaystyle \frac{1}{x+i}$ and
$\displaystyle \frac{1}{i+2-x} < \beta_i < \displaystyle \frac{1}{i-x}$.
From $(\ref{4.22})$-$(\ref{4.25})$ and $(\ref{4.15})$, we have
\begin{equation}\label{4.26}
\begin{split}
h'_{n+1} (x) &= g \left\{ -h'_n(\gamma)\frac{1-x}{(x+1)(3-x)} \right. \\
             &-\sum_{i=1,3,\ldots} \frac{2(1-x)}{(x+G+i)(x+i)(x+i+2)} h'_n(\alpha_i) \\
             &-\sum_{i=3,5,\ldots} \frac{2(1-x)}{(i-x+G)(i+2-x)(i-x)} h'_n(\beta_i) \\
             &+ \left(G^2-(1-x)^2\right) \left( \sum_{i=1,3,\ldots} \frac{1}{(x+G+i)^2(x+i)(x+i+2)} h'_n(\alpha_i) \right. \\
             &- \left. \left.\sum_{i=3,5,\ldots} \frac{1}{(i-x+G)^2(i+2-x)(i-x)} h'_n(\beta_i) -\frac{1}{2} S_{5} +\frac{1}{2} S_{6}\right) \right\},
\end{split}
\end{equation}
where $\displaystyle\frac{1}{3-x} < \gamma < \displaystyle\frac{1}{x+1}$.
Now, for $x=1$ we have $\alpha_i = \beta_{i+2}$, and then $h'_{n+1}(1) = 0$.
By the Mean Value Theorem, we know there exists at least one $c \in (0, 1)$ such that
\begin{equation}\label{4.27}
h'_{n+1} (1) - h'_{n+1} (0) = h''_{n+1} (c),
\end{equation}
i.e., $- h'_{n+1} (0) = h''_{n+1} (c)$.
If $h'_{n+1}(0) < 0$, then $h''_{n+1} (c) > 0$. Since $h'_{n+1} \neq 0$ on $(0, 1)$ and $h'_{n+1}$ is increasing in a neighbourhood of $c$,
it follows that $h'_{n+1}$ is increasing on $[0, 1]$.
Similarly, if $h'_{n+1}(0) > 0$, then $h''_{n+1} (c) < 0$, and it follows that $h'_{n+1}$ is non-increasing on $[0, 1]$.
In both cases it results that
\begin{equation}\label{4.28}
M_{n+1} = \max_{x \in [0, 1]} \left| h'_{n+1} (x) \right| = \left| h'_{n+1} (0) \right|.
\end{equation}
For, $x=0$, we have $\alpha_i = \beta_{i}$, $\gamma = \alpha_1$ and we obtain
\begin{equation} \label{4.29}
h'_{n+1}(0) = -4 \cdot g \cdot \sum_{i=1,3,5, \ldots} \frac{h'_n(\alpha_i)}{(G+i)i(i+2)}.
\end{equation}
Thus,
\begin{equation}\label{4.30}
M_{n+1} \leq 4 \cdot g \cdot M_n \cdot \sum_{i=1,3,5, \ldots} \frac{1}{(G+i)i(i+2)}
\end{equation}
and the proof is complete.
\end{proof}

\noindent \textbf{Proof of Theorem \ref{Th.GKL}.}
For $\{H_n\}$ in (\ref{4.2}), we introduce a function $R_n(x)$ such that
\begin{equation}
H_{n} (x) = \frac{1}{3 \log G} \log \frac{(G+1)(G-1+x)}{(G-1)(G+1-x)} + R_n(x). \label{4.31}
\end{equation}
Because $H_n(0)=0$ and $H_n(1)=1$, we have $R_n(0)=R_n(1)=0$.
To prove Theorem \ref{Th.GKL}, we have to show the existence of a constant
$0 < \eta < 1$ such that
\begin{equation}\label{4.32}
R_n(x) = \mathcal{O}(\eta^n).
\end{equation}
For $\{h_n\}$ in (\ref{4.6}), if we can show that $h_n(x) = \frac{1}{3 \log G} + \mathcal{O}(\eta^n)$,
then integrating (\ref{4.6}) will show (\ref{4.3}).
To demonstrate that $\{h_n\}$ has this desired form, it suffices to prove the following lemma.
\begin{lemma} \label{lem.4.5}
For any $x \in [0, 1]$ and $n >0$ there exists a constant $\eta:=\eta(x)$ with $0 < \eta < 1$ such that
\begin{equation}
h'_{n}(x) = {\mathcal O}(\eta^n). \label{4.33}
\end{equation}
\end{lemma}
\begin{proof}
\, Let $\eta$ be as in Lemma \ref{lem.4.4}.
Using this lemma, to show (\ref{4.33}) it is enough to prove that $\eta < 1$.
Calculating the sum of the series involved, we obtain
$
\eta= 4 \cdot g \cdot 0.150853 = 0.372929.
$
\end{proof}


\end{document}